\documentclass[11pt]{amsart}

\usepackage{amsmath}
\usepackage{amssymb}
\usepackage{amsfonts}
\usepackage{amsthm}
\usepackage{enumerate}
\usepackage[mathscr]{eucal}

\usepackage{graphicx}
\usepackage{verbatim}

\newcommand{\sK}{\mathscr K}

\newcommand{\ci}[1]{_{{}_{\!\scriptstyle{#1}}}}

\newcommand{\Be}{\begin{equation}}
\newcommand{\Ee}{\end{equation}}

\newcommand{\Bm}{\begin{multline}}
\newcommand{\Em}{\end{multline}}
\newcommand{\Bea}{\begin{eqnarray}}
\newcommand{\Eea}{\end{eqnarray}}
\newcommand{\Beas}{\begin{eqnarray*}}
\newcommand{\Eeas}{\end{eqnarray*}}
\newcommand{\Benu}{\begin{enumerate}}
\newcommand{\Eenu}{\end{enumerate}}
\newcommand{\Bi}{\begin{itemize}}
\newcommand{\Ei}{\end{itemize}}

\def\intslash{\rlap{\kern  .32em $\mspace {.5mu}\backslash$ }\int}
\def\qsl{{\rlap{\kern  .32em $\mspace {.5mu}\backslash$ }\int_{Q_x}}}

\def\vth{\vartheta}

\def\Q{\mathcal Q}

\def\emph#1{{\it #1 }}
\def\diam{{\text{\it  diam}}}

\def\supp{{\text{\rm supp}}}

\def\inn#1#2{\langle#1,#2\rangle}
\def\biginn#1#2{\big\langle#1,#2\big\rangle}

\def\noi{\noindent}

\def\meas{{\text{\rm meas}}}

\def\card{\text{\rm card}}
\def\lc{\lesssim}
\def\gc{\gtrsim}

\def\eps{\varepsilon}

\def\la{\lambda}             \def\La{\Lambda}

\def\fQ{{\mathfrak {Q}}}

\def\fS{{\mathfrak {S}}}

\def\bbN{{\mathbb {N}}}

\def\bbR{{\mathbb {R}}}

\def\bbZ{{\mathbb {Z}}}

\def\cA{{\mathcal {A}}}

\def\cC{{\mathcal {C}}}

\def\cE{{\mathcal {E}}}
\def\cF{{\mathcal {F}}}

\def\cS{{\mathcal {S}}}
\def\cT{{\mathcal {T}}}

\def\Q{{\hbox{\bf Q}}}

 at 10 true pt

\def\be#1{\begin{equation}\label{ #1}}
\def\endeq{\end{equation}}
\def\endal{\end{align}}
\def\bas{\begin{align*}}
\def\eas{\end{align*}}
\def\bi{\begin{itemize}}
\def\ei{\end{itemize}}

\def\eps{\varepsilon}

\def\emph#1{{\it #1}}
\def\textbf#1{{\bf #1}}

\theoremstyle{plain}
   \newtheorem{theorem}{Theorem}[section]

   \newtheorem{proposition}[theorem]{Proposition}
   
   \newtheorem{lemma}[theorem]{Lemma}

   \newtheorem{theorem*}{Theorem}

\theoremstyle{remark}

\theoremstyle{definition}

\numberwithin{equation}{section}

\begin{document}

\title{A weak type   bound for a  singular integral}
\author{Andreas Seeger}

\address{
Department of Mathematics\\ University of Wisconsin-Madison\\Madison, WI 53706, USA}


\begin{thanks} {Supported in part by NSF grant 1200261}
\end{thanks}

\begin{abstract} A weak type $(1,1)$ estimate is established for the
first order  $d$-commutator
introduced by Christ and Journ\'e, in dimension $d\ge 2$.
\end{abstract}

\maketitle

\section{Introduction}

Let $K$ be  regular Calder\'on-Zygmund convolution kernel on $\bbR^d$, $d\ge 2$,
{\it i.e.} $K\in \cS'$,  locally bounded
in $\bbR^d\setminus \{0\}$ and
  satisfies
\Be\label{czsize}
|K(x)|\le A|x|^{-d}
 \quad x\neq 0,
\Ee
and,  for some $\eps\in (0,1]$,
\Be \label{czreg}
 |K(x+h)-K(x)| \le A
|h|^\eps |x|^{-d-\eps}  \text{ if } |x|>2|h|;
\Ee
moreover $$\|\widehat K\|_\infty\le A<\infty.$$
Let $a\in L^\infty(\bbR^d)$.
The so-called {\it $d$-commutator} $T\equiv T[a]$ of first order associated with $K$ and $a$ is defined
for Schwartz functions $f$ by
$$T[a]f(x) = p.v.  \int K(x-y) \int_0^1 a(sx+(1-s) y) ds\,
  f(y) dy\,.
$$
In dimensions $d\ge 2$   this definition yields  a rough analog
of the Calder\'on commutator \cite{cald} in one dimension.
Christ and Journ\'e \cite{ch-j} proved that  $T$ and higher order versions extend to  bounded operators
 on $L^p(\bbR^d)$, for  $1<p<\infty$.
We prove  that the first order $d$-commutator is  also of weak type $(1,1)$.


\begin{theorem} \label{maintheorem} There is $C_{d}<\infty$ so that for any  $f\in L^1(\bbR^d)$ and any  $a\in L^\infty(\bbR^d)$,
$$\sup_{\la>0} \la\,
\meas\big(\{x\in \bbR^d:|T[a]f(x)|>\la\}\big) \le C_{d} A \tfrac{1}{\eps}\log(\tfrac 2\eps) 
 \|a\|_\infty \|f\|_{L^1(\bbR^d)}\,.
$$
\end{theorem}




In two dimensions this result has recently been established by
Grafakos and Honz\'ik \cite{grafakos-honzik} (assuming $\eps=1$).
Their approach relies on a method developed in \cite{ch1}, \cite{ch-rdf} and \cite{hofmann} for proving  a weak type $(1,1)$ bound for rough singular convolution operators.  A dyadic decomposition $T[a]=\sum T_j$ is used on the kernel side, and the argument relies on the fact  that in two dimensions
the kernels of  the operators  $T_j^*T_i$ have certain
 H\"older continuity properties. This argument is no longer valid in
higher dimensions. It is conceivable that for  $d\ge 3$    
one might be able to develop  the more complicated  iterated $T^*T$ arguments
introduced by Christ and Rubio de Francia \cite{ch-rdf} and further extended  by Tao \cite{tao}, but
this route would lead to   substantial technical difficulties and we
 shall not pursue it.
Our approach is different and
relies on an idea
introduced in \cite{see}.  An  orthogonality argument
for a  microlocal decomposition of the operator  is  used.
The implementation of this idea in the present setting is
more complicated  in the convolution case  as the Christ-Journ\'e
operators  can be viewed  as an amalgam of
operators of generalized convolution type  (for which there is a
suitable calculus of wavefront sets)
 and operators of
multiplication  with a  rough function.


\noindent{\it Notation.} 
We write $\cE_1\lc \cE_2$ to indicate that $\cE_1\le C_d \cE_2$
for some \lq constant\rq \  $C$ that may  depend on $d$. We also use the notation $\lc_N$ to indicate dependence on other parameters  $N$. We denote by $\widehat f$ or $\cF f$ the Fourier transform of $f$, defined for Schwartz functions by 
 $\widehat f(\xi)= \int f(y)e^{-i\inn y\xi}dy$.

\noindent{\it This paper.} 
In  \S\ref{decompsect} we outline the proof of Theorem \ref{maintheorem}
with three technical  propositions 
\ref{secjunkL1},
\ref{L1prop},  \ref{L2prop} proved in 
\S \ref{easyest}, \S\ref{proofofL1prop}, \S\ref{proofofL2prop}, 
respectively.  In \S\ref{openproblems} we shall mention some open problems.

\section{Decompositions and auxiliary estimates}
\label{decompsect}
We may assume that $A\le 1$, $\|a\|_\infty\le 1$ and  write $T=T[a]$.
Fix $f\in L^1(\Bbb R^d)$.
We use the standard  Calder\'on-Zygmund decomposition of $f$ at
height $\la$ (see \cite{stein-si}). Then
$$f\,=\,g+b\,=\,g+\sum_{Q\in \fQ_\la} b_Q$$
where $\|g\|_\infty\le\la$, $\|g\|_1\lc\|f\|_1$, each
$b_Q$ is supported in a dyadic cube $Q$ with sidelength $2^{L(Q)}$ and center $y_Q$,
and $\fQ_\la$ is a family of dyadic cubes with  disjoint interiors.
Moreover $\|b_Q\|_1\lc\la |Q|$ for each $Q\in \fQ_\la$ and
$\sum_{Q\in \fQ_\la}|Q|\lc \la^{-1}\|f\|_1$. For each $Q$ let $Q^*$ be the
dilate of $Q$ with same center and $L(Q^*)=L(Q)+10$,
and let $E=\bigcup_{Q\in \fQ_\la} Q^*$. Then also
$$\meas(E)\lc\la^{-1} \|f\|_1.$$
Finally, for each $Q$, the mean value of $b_Q$ vanishes:
$$\int b_Q(y)dy=0.$$

Since $T$ is bounded on $L^2$ (\cite{ch-j}) we have, as in  
standard Calder\'on-Zygmund theory,
the estimate for the good function $g$
$$
\|Tg\|_2^2\le \|T\|_{L^2\to L^2}^2\|g\|_2^2\lc \|g\|_1 \|g\|_\infty\lc
\la\|g\|_1
$$
and by Tshebyshev's inequality,
$$
\big|\{x\in\Bbb R^d: |Tg(x)|>\la/10\}\big|\,\
\le \,100\la^{-2}\|Tg\|_2^2\,\lc \la^{-1}\|g\|_1\,\lc\,\la^{-1}\|f\|_1.
$$

We use a dyadic decomposition of the kernel.
Let $\varphi$ be a radial $C^\infty$ function, so that $\varphi(x)=1$ for $|x|\le 1$ and $\varphi(x)=0$
for $|x|\ge 6/5$.
Let  $$K_j(x)= \big(
\varphi(2^{-j}x)-\varphi(2^{-j+1}x)\big)
K(x) $$
so that $K=\sum K_j$ in the sense of distributions on $\bbR^d\setminus \{0\}$ and $K_j$ is supported in the annulus
$\{x: 2^{j-1}\le |x|\le \tfrac 65  2^j\}$.
Let  $T_j$ be the integral operator with Schwartz kernel
$$K_j(x-y) \int_0^1 a(sx+(1-s) y ) \,ds\,.$$
For $m\in\Bbb Z$ let
$$B_m=\sum_{\substack{Q\in \fQ_\la\\L(Q)=m}}b_Q.$$
Observe that
for each $j$, $m$ the function $T_j B_m$ belongs to $L^1$, and that
$$\supp(T_j B_m) \subset E, \quad  m\ge j.$$
Moreover, for each $n$,
 $$\sum_j \|T_j B_{j-n}\|_1 \lc \|f\|_1$$
and thus, if
$$n(\eps)= 10^{10}d\eps^{-1}\log_2(2\eps^{-1})
$$
we have 
by Tshebyshev's inequality
\begin{multline} \label{firstterms}
\meas\,\big(\big\{x\in \bbR^d: \sum_{0<n\le n(\eps)} \sum_j|T_j B_{j-n}(x)| 
 >\la/10\big\}\big)\\ \lc \,\eps^{-1} \log(2\eps^{-1}) \la^{-1}\|f\|_1.
\end{multline}
It thus suffices  to show that
$\sum_{n>n(\eps)}  (\sum_j T_j B_{j-n})$ converges in the topology of
 $(L^1+L^2)(\bbR^d\setminus E)$
and satisfies the inequality
\Be \label{mainineqoff}
\meas\,\big( \big\{x\in \bbR^d\setminus E: 
\sum_{n>n(\eps)} \big|\sum_jT_j B_{j-n}(x)\big|  
>4\la/5\big\}\big)\, \lc\, \la^{-1}\|f\|_1
\Ee

\subsection*{Finer decompositions}
We first slightly modify the kernel $K_j$ and subtract   an acceptable error
term which is small in $L^1$. In what follows assume $n>n(\eps)$ as defined above.
Let \Be\label{elln}\begin{aligned}
\ell(n)&=[2\log_2(n)]+2
\\
\ell_\eps(n)&= [2\eps^{-1} \log_2 n]+2.
\end{aligned}
\Ee
Let $\Phi$ be a radial $C^\infty_0$ function supported in $\{|x|\le 1\}$,
and satisfying $\int\Phi(x) dx=1$. Let $\Phi_m(x)= 2^{-md}\Phi(2^{-m}x)$.
Define  $$K_j^n= K_j * \Phi_{j-\ell_\eps(n)}\,.$$
Then $K_j^n$ is supported in $\{x:2^{j-2}\le |x|\le 2^{j+2}\}$, and, by the regularity assumption \eqref{czreg},
\begin{align}\notag
 \|K_j -  K_j^n\|_1 &\lc 2^{-(j-\ell_\eps(n))d}
\iint\limits_{\substack{|h|\le 2^{-(j-1-\ell_\eps(n))}\\2^{j-2}\le|x|\le 2^{j+2}}}
|K_j(x)-K_j(x-h)| \,dx\,dh
\\\label{KjminusKjn} &\lc 2^{-\ell_\eps(n)\eps} \lc n^{-2}\,.
\end{align}
By differentiation and \eqref{czsize}
\Be \label{Kjnder}
|\partial^\alpha K^n_j (x)|\le C_\alpha 2^{-jd} 2^{(\ell_\eps(n)-j)|\alpha|} \,.
\Ee
Let $\vth_n \in C^\infty(\bbR)$ be supported in $(n^{-2}, 1-n^{-2})$, such that
$\vth_n(s)=1$ for $s\in [2n^{-2}, 1-2n^{-2}]$, and such that the derivatives of
$\vth_n$ satisfy the natural estimates
\Be\label{thetaest} \|\vth_n^{(N)}\|_\infty \le C_N n^{2N}\,.
\Ee
We then let $T^n_j$
be the integral operator with Schwartz kernel
$$K_j^n(x-y) \int \vth_n(s) a(sx+(1-s) y)\, ds\,.$$
The following lemma is an immediate consequence of estimate \eqref{KjminusKjn} and the support property of $\vth_n$. 
\begin{lemma} \label{junkL1} The operator $T_j-T^n_j$ is bounded on $L^1$, with operator norm
$$\|T_j-T_j^n\|_{L^1\to L^1}\lc n^{-2}\,.$$
\end{lemma}
The lemma implies
\begin{align*}
&\meas\,\big(\big\{x: \sum_{n>n(\eps)} \big|\sum_j(T_j B_{j-n}(x)-
T_j^n B_{j-n}(x))\big|  >\la/10\big\}\big)
\\
&\le\,10\la^{-1} \Big\|\sum_{n>n(\eps)}\sum_j  |T_j B_{j-n}-T_j^n B_{j-n}| \Big\|_1
\\
&\lc \,\la^{-1} \sum_{n\ge 1} n^{-2} \sum_j\|B_{j-n}\|_1
 \lc\, \la^{-1}\|f\|_1
\end{align*}
and therefore it is enough to show
\Be \label{mainineqoff2}
\meas\,\big(\big\{x: \sum_{n>n(\eps)} \sum_j|T_j^n B_{j-n}(x)|  > \tfrac {7}{10}\la\big\}\big) \lc \la^{-1}\|f\|_1\,.
\Ee

For the proof of \eqref{mainineqoff2} we
subtract various regular or small terms from
the operators $T_j^n$.
Let $\ell(n)$ be as  in \eqref{elln} and denote by $P_m$ the convolution operator with convolution kernel $\Phi_m$ (defined following \eqref{elln}).
We have
\begin{proposition} \label{secjunkL1}
For $n>1$,
$$\|P_{j-n+\ell(n)}
T^n_j B_{j-n}
\|_1 \lc n^{-2}\log n \|B_{j-n}\|_1\,.$$
\end{proposition}
The proposition will be proved in \S\ref{easyest}. It  yields
\begin{align*}
&\meas\,\big(\big\{x\in \bbR^d\setminus E: \sum_{n>n(\eps)} \big|\sum_j P_{j-n+\ell(n)}
T_j^n B_{j-n}(x))\big|  >\la/10\big\}\big)
\\
&\lc\,10 \la^{-1}\sum_{n>n(\eps)} \sum_j \| P_{j-n+\ell(n)}
T^n_j B_{j-n} \|_1\\
&\lc \,\la^{-1} \sum_{n>1} n^{-2}\log n \sum_j\|B_{j-n}\|_1
 \lc \la^{-1}\|f\|_1
\end{align*}
and thus it remains to consider the term
\Be \label{mainterm}\sum_{n>n(\eps)}\sum_j (I-P_{j-n+\ell(n)} )
T_j^n B_{j-n}(x)\Ee  and to estimate the measure of the set where
$|\text{\eqref{mainterm}}|> 3\la/5$.
We shall need to exploit the fact that the integral
$\int_0^1 a(sx+(1-s)y) ds$ smoothes the rough function $a$ in the direction parallel to $x-y$, and use a microlocal decomposition which we now describe.

Let $1/10
<\gamma<9/10$ (say $\gamma=1/2$),
and let $\Theta_n$ be set of unit vectors
with the property that if $\nu\neq \nu'$, $\nu,\nu'\in \Theta_n$ then
$|\nu-\nu'|\ge  2^{-4-n\gamma}$, and assume that $\Theta_n$ is {\it maximal} with respect to this property. Note that
$$\card(\Theta_n)\lc 2^{n\gamma(d-1)}\,.$$ For each $\nu$ we may choose a function 
$\widetilde \chi_{n,\nu}$ 
 on $C^\infty(S^{d-1})$ with the property that
$\widetilde \chi_{n,\nu}(x)\ge 0$, $\widetilde \chi_{n,\nu}(\theta)=1$ if $|\theta-\nu|\le  2^{-3-n\gamma}$,
$\widetilde \chi_{n,\nu}(\theta)=0$ if $|\theta-\nu|>  2^{-2-n\gamma}$,
and such that for each $M\in \bbN$  the functions 
$2^{-n\gamma M}\widetilde \chi_{n,\nu}$ form a bounded family in $C^M(S^{d-1})$.
For each $\theta$ there is at least one $\nu$ such that 
$\widetilde \chi_{n,\nu}(\theta)=1$, 
by the maximality assumption, moreover by the separatedness  assumption the number of $\nu\in \Theta_n$ for which $\widetilde \chi_{n,\nu}(\theta)\neq 0$ is bounded above, uniformly in $\theta$ and $n$. 
Define, for $\nu\in \Theta_n$
$$\chi_{n,\nu}(x)= \frac{\widetilde \chi_{n,\nu}(\tfrac{x}{|x|})}
{\sum_{\nu'\in \Theta_n}\widetilde \chi_{n,\nu'}(\tfrac{x}{|x|})}\,.
$$
Then $\sum_{\nu\in \Theta_n} \chi_{n,\nu}(x)=1$
for every $x\in \bbR^d\setminus\{0\})$ and by homogeneity we have the following estimates
for multiindices $\alpha$ and $x\neq 0$,
\begin{align*}
|(\inn{\nu}{\nabla})^M \chi_{n,\nu}(x)| &\le C_M |x|^{-M}\,,
\\|\partial^\alpha \chi_{n,\nu}(x)|\,&\le\,  C_\alpha 2^{n\gamma|\alpha|}|x|^{-|\alpha|}\,.
\end{align*}

Let $K_j^{n,\nu}(x)= K_j^n(x) \chi_{n,\nu}(x)$ and let $T_j^{n,\nu}$ be the operator with Schwartz kernel
$$K_j^{n,\nu}(x-y) \int \vth_n(s)\, a(sx+(1-s) y)\, ds\,.$$
We then have $$T^n_j=\sum_{\nu\in \Theta_n} T^{n,\nu}_j\,.$$
Let $\phi\in C^\infty(\bbR)$ so that $\phi(u)=1$ for $|u|<1/2$ and $\phi(u)=0$ for
$|u|\ge 1$ and define the singular convolution operator $\fS_{n,\nu}$ by
$$\widehat {\fS_{n,\nu} f}(\xi)= \phi\big(2^{n\gamma} n^{-5} \inn{\nu}{\tfrac{\xi}{|\xi|}}\big)\widehat f(\xi).$$

The terms involving $(I-\fS_{n,\nu})T_j^{n,\nu} $ can be dealt with by $L^1$ estimates. In \S\ref{proofofL1prop} we shall prove

\begin{proposition} \label{L1prop}
For $n>n(\eps)$, $\nu \in \Theta_n$,
$$\Big\|\sum_{j} 
(I-P_{j-n+\ell(n)} ) (I-\fS_{n,\nu})T_j^{n,\nu} 
B_{j-n}\Big\|_1 \lc
n^{-2} 2^{-n\gamma(d-1)} \|f\|_1\,.
$$
\end{proposition}

For the rougher terms  involving
$\fS_{n,\nu} T_j^{n,\nu} $ we shall prove  
in \S\ref{proofofL2prop}
the following $L^2$ estimate.

\begin{proposition} \label{L2prop}
For $n>n(\eps)$,
$$\Big\|\sum_{\nu\in \Theta_n} \sum_j 
(I-P_{j-n+\ell(n)} ) \fS_{n,\nu} T_j^{n,\nu} B_{j-n}\Big\|_2^2 \lc
2^{-n\gamma} n^5 \la \|f\|_1\,.
$$
\end{proposition}

Given the propositions we can finish the outline of the proof of Theorem
\ref{maintheorem}.
Namely by Tshebyshev's inequality,
\begin{align*}
\meas\,&\big(\big\{ x: \big|\sum_{n>n(\eps)}\sum_j (I-P_{j-n+\ell(n)} )
T_j^n B_{j-n}(x)\big| >\frac{3}{5}\la\big\}\big)
\\ &\lc\,
5\la^{-1}\Big\|\sum_{n>n(\eps)}\sum_{\nu\in \Theta_n}\sum_{j} (I-P_{j-n+\ell(n)} )  (I-\fS_{n,\nu})
T_j^{n,\nu} B_{j-n}\Big\|_1
\\
&+\,25\la^{-2}\Big\|\sum_{n>n(\eps)}\sum_{\nu\in \Theta_n}\sum_{j} 
(I-P_{j-n+\ell(n)} ) \fS_{n,\nu} T_j^{n,\nu} B_{j-n}\Big\|_2^2
\end{align*}
and by the propositions and Minkowski's inequality  this is bounded by
a constant times
$$
\la^{-1}\|f\|_1
\Big( \sum_{n} n^{-2} 2^{-n\gamma(d-1)} \card (\Theta_n)
+
\sum_{n}2^{-n\gamma}n^5 \Big)\lc  \la^{-1}\|f\|_1\,.
$$

\section{Proof of Proposition \ref{secjunkL1}}
\label{easyest}
Let $Q\in \fQ_\la$  with $L(Q)=j-n$.
We apply Fubini's theorem
and write
\begin{align*}
&P_{j-n+\ell(n)} T^n_j b_Q(x) \,=\, \int\vth_n(s)\,\int b_Q(y) \times \\
&\qquad\Big[\int\Phi_{j-n+\ell(n)}(x-w) K^n_j(w-y)
a(sw+(1-s)y)\, dw\Big]\, dy \, ds\,.
\end{align*}
Changing  variables
$z= w+ \tfrac{1-s}{s} y$ we get
$$P_{j-n+\ell(n)} T^n_j b_Q(x)
=\int\vth_n(s)\int a(sz) \int \cA^{x,z,s}_{j,n}(y)
b_Q(y) \, dy
\, dz\, ds
$$
where
$$\cA^{x,z,s}_{j,n}(y)=
\Phi_{j-n+\ell(n)}(x-z+\tfrac{1-s}{s} y) K^n_j(z-\tfrac{y}{s}). $$

We expand $\cA^{x,z,s}_{j,n}(y)$ about the center $y_Q$ of $Q$ and in view of the
cancellation of $b_Q$
 we  may   write
\begin{multline*}
|P_{j-n+\ell(n)}
T^n_j b_Q(x)| \\
\le \iint |\vth_n(s) a(sz)|\,\Big|\int
\big(\cA^{x,z,s}_{j,n}(y)-\cA^{x,z,s}_{j,n}(y_Q)\big)
b_Q(y) \, dy \Big| dz\, ds\,.
\end{multline*}
Using  $$\cA^{x,z,s}_{j,n}(y)-\cA^{x,z,s}_{j,n}(y_Q)\,=\,
\biginn{y-y_Q}{\int_0^1 \nabla \cA^{x,z,s}_{j,n}(y_Q+\sigma(y-y_Q)) \, d\sigma}
$$ in the previous display one  obtains
after applying Fubini's theorem
\begin{align*}
&\|P_{j-n+\ell(n)}
T^n_j b_Q(x)\|_1
\le \diam(Q)\, \int_0^1 \int |\vth_n(s)|
\times
\\ &\qquad\Big[
\|\nabla \Phi_{j-n+\ell(n)}\|_1
\frac{1-s}{s}
\int |b_Q(y)| 
\int \big|K^n_j(z-\tfrac{y_Q+\sigma(y-y_Q)}{s})\big|\, dz \, 
dy\,
 \\&\quad +\| \Phi_{j-n+\ell(n)}\|_1
\int |b_Q(y)|  \int \frac{1}{s}\big|\nabla K^n_j(z-\tfrac{y_Q+\sigma(y-y_Q)}{s})\big| \, dz\,
dy \Big]\,  ds\, d\sigma\,.
\end{align*}
Now use 
$\|\nabla K^n_j\|_1\lc 2^{-j+\ell_\eps(n)}$
and $\int_0^1|\vth_n(s)| s^{-1} ds \lc \log n$, and since $\diam (Q)\lc 2^{j-n}$  we obtain
\begin{align*}
\big\|P_{j-n+\ell(n)}
T^n_j b_Q\big\|_1
&\lc \log n\,\big[ 2^{-\ell(n)}+2^{\ell_\eps(n)-n} \big] \|b_Q\|_1
\\
&\lc n^{-2} \log n\,\|b_Q\|_1.
\end{align*}
Finally we  sum over all $Q\in \fQ_\la$ with $L(Q)=j-n$ to obtain the asserted bound.
 \qed




\section{Proof of Proposition \ref{L1prop}}\label{proofofL1prop}
Let $Q\in \fQ_\la$ with $L(Q)=j-n$, and let $y_Q$ be the center of $Q$.
Fix a unit vector $\nu$, and let $\pi_\nu^\perp$ be the projection to the orthogonal complement of $\nu$, i.e. $\pi_\nu^\perp(x)=x-\inn{x}{\nu}\nu$.
In view of the support properties of the kernel it suffices to show that
for $n>n(\eps)$
\Be \label{fixnuest}\Big\|(I-P_{j-n+\ell(n)} ) (I-\fS_{n,\nu})
T_j^{n,\nu} b_Q\Big\|_1 \lc
n^{-2}  2^{-n\gamma(d-1)} \|b_Q\|_1\,,
\Ee
 under  the additional assumption that  the support of $a$ is contained in
$$
\big\{ y: |\inn{y-y_Q}{\nu}|\le 2^{j+4}d\,,\,
|\pi_\nu^\perp(y-y_Q)|\le 2^{j+4-n\gamma}d\,
\big \}\,.$$
Note that with this hypothesis
\Be \label{FTaest}\|\widehat a\|_\infty \lc 2^{jd-n\gamma(d-1)}\,.\Ee

We introduce a frequency  decomposition of $a$. Let $\varphi$ be a radial $C^\infty$ function as in \S\ref{decompsect}, but now defined
in $\xi$-space, so that $\varphi(\xi)=1$ for $|\xi|\le 1$
and $\varphi(\xi)=0$
for $|\xi|\ge 6/5$. Define $\beta_k(\xi)=
\varphi(2^k\xi)-\varphi(2^{k+1}\xi)$; then  $\beta_k$ is supported in $\{\xi: 2^{-k-1}\le|\xi|\le  \tfrac 65 2^{-k}\}$.
Let $\widetilde \beta$ be a radial
$C^\infty$ function
so that $\widetilde \beta$ is supported in $\{\xi: 1/3\le |\xi|\le 3/2\}$ and
$\widetilde \beta(\xi)=1$ for $1/2\le |\xi|\le 6/5$, and define
$\widetilde \beta_k(\xi)= \widetilde \beta(2^k\xi)$. Then $\beta_k\widetilde\beta_k=\beta_k$.
Define convolution operators $V_k$,
$\La_k$,
$\widetilde \La_k$ with Fourier multipliers
$\varphi(2^k\cdot)$,
$\beta_k$, $\widetilde \beta_k$, respectively;
then
$\La_k\widetilde \La_k=\La_k$ and, for every $m\in \bbZ$,  the identity operator is decomposed as
$I=V_m+\sum_{k<m} \La_k$.

For fixed $y\in Q$ we define an operator
$\sK^{n,\nu}_{j,y} $  acting on $a$  by
$$\sK^{n,\nu}_{j,y}[a](x) =
K^{n,\nu}_j(x-y) \int\vth_n(s) a(sx+(1-s)y) ds\,
$$
so that
\Be\label{integralbQ}T^{n,\nu}_{j}b_Q(x)= \int b_Q(y) \sK^{n,\nu}_{j,y}[a](x)\, dy\,.\Ee

We use dyadic frequency decompositions and split
\begin{multline}\label{split}
(I-\fS_{n,\nu})(I-P_{j-n+\ell(n)} ) T_j^{n,\nu} b_Q =\\
\sum_{k_1} \La_{k_1}
(I-\fS_{n,\nu}) \widetilde \La_{k_1}(I-P_{j-n+\ell(n)} )
\int b_Q(y) \sK_{j,y}^{n,\nu} [a]\, dy
\end{multline}
and then further split in \eqref{split}
\Be\label{asplit}a= V_{j-n+\ell(n)}a+\sum_{k_2<j-n+\ell(n)} \La_{k_2}a\,.
\Ee

We  prove three  lemmata with various bounds for the terms in
\eqref{split}, \eqref{asplit}.

\begin{lemma} \label{firstcancel}
$$\Big
\|\int b_Q(y) \sK^{n,\nu}_{j,y}[V_{j-n+\ell(n)}a]\,dy\Big\|_1 \lc  n^{-2} 2^{-n\gamma(d-1)}
 \|b_Q\|_1\,.
$$
\end{lemma}
\begin{proof}
We use the cancellation of $b_Q$ to estimate the left-hand side by
$$\int|b_Q(y) |
\int|\sK^{n,\nu}_{j,y}[V_{j-n+\ell(n)}a](x)
-\sK^{n,\nu}_{j,y_Q}[V_{j-n+\ell(n)}a](x) | \,dx  \, dy\,.$$
For $y\in Q$ we may estimate
$$\int |\sK^{n,\nu}_{j,y}[V_{j-n+\ell(n)}a](x)
-\sK^{n,\nu}_{j,y_Q}[V_{j-n+\ell(n)}a](x) | \,dx  \le \cE_1(y)+\cE_2(y)
$$
where
$$\cE_1(y)=
\|V_{j-n+\ell(n)}a\|_\infty
\int|K^{n,\nu}_{j}(x-y) -K^{n,\nu}_{j}(x-y_Q)|\,dx
$$ and, abbreviating
\begin{multline*}
\Gamma^Q_{j-n+\ell(n)}
(x,y,z)\,=\,
\\
\int_{0}^1 \biginn{y-y_Q}{\nabla\cF[\varphi(2^{j-m+\ell(n)}\cdot)] 
(sx+(1-s)(y_Q+\sigma(y-y_Q))-z)}\,d\sigma,
\end{multline*}
$\cE_2$ is given by
$$\cE_2(y)=
\int |K^{n,\nu}_j(x-y_Q)| \int|\vth_n(s)|   \int |a(z)|\,|\Gamma^Q_{j-n+\ell(n)}
(x,y,z)| \, dz\, ds\, dx.
$$

Now by \eqref{Kjnder},  and since  $|\partial_x \chi_{n,\nu}(x)|\lc 2^{n\gamma}|x|^{-1}$
we get
$$
|\cE_1(y)|\le |y-y_Q|\|\nabla K^{n,\nu}_j\|_1
\lc 2^{j-n}[ 2^{\ell_\eps(n)-j}+ 2^{n\gamma-j}] 2^{-n\gamma(d-1)}.
$$
Notice that for $n>n(\eps)$ and $\gamma> 1/10$ we have 
$2^{\ell_\eps(n)} \lc 2^{n\gamma}$ and thus we see that
$|\cE_1(y)|\lc 2^{-n\gamma(d-1)} n^{-2}$.
Moreover, with  $\chi_k:=\cF^{-1}[\varphi(2^k\cdot)]$,
$$
|\cE_2(y)|\lc \|K^{n,\nu}_j\|_1 |y-y_Q|\|\nabla \chi_{j-n+\ell(n)}\|_1 \lc
2^{-n\gamma(d-1)} 2^{j-n} 2^{n-j-\ell(n)}
$$
which is 
$\lc 2^{-n\gamma(d-1)} n^{-2}$. Integrating in $y$, we get
$$\int \big(|\cE_1(y)|+|\cE_2(y)|\big)\,|b_Q(y)| dy 
\,\lc\,2^{-n\gamma(d-1)} n^{-2} \|b_Q\|_1,$$
and the assertion follows.
\end{proof}

\begin{lemma} \label{VandLP}
Let $y\in Q$ and  $a$ be  as in \eqref{FTaest}.

(i) Let $k_1> k_2+\ell(n)+10$. Then
$$\big\|\La_{k_1}\sK^{n,\nu}_{j,y}
[\La_{k_2}a]
\big\|_1 \\
\le  C_N 2^{-n\gamma(d-1)} \min\{1,
  n^{2d+2N}2^{n\gamma} 2^{(k_2-j+n\gamma)N}\}
$$

(ii) Let  $k_1<k_2-10$. Then
\begin{multline*}\big\|\La_{k_1}\sK^{n,\nu}_{j,y}
[\La_{k_2}a]
\big\|_1 \,+\,\big\|\La_{k_1}\sK^{n,\nu}_{j,y}
[V_{k_2}a]
\big\|_1 \\
\le  C_N 2^{-n\gamma(d-1)} \min\{1,
  2^{n\gamma} 2^{(k_1-k_2)d} 2^{(k_1-j+n\gamma)N} \}\,.
\end{multline*}
\end{lemma}

\begin{proof}
Clearly $\|\sK^{n,\nu}_{j,y}[a]\|_1\lc 2^{-n\gamma(d-1)} \|a\|_\infty$, and since the operators $\La_k$, $V_k$ are uniformly bounded we get the bound
$O(2^{-n\gamma(d-1)})$ in (i) and (ii).
We seek to prove the two other  bounds for $\La_{k_1}\sK^{n,\nu}_{j,y}
[\La_{k_2}a]$
under the assumptions
$k_1<k_2-10$, and  $k_1>k_2+\ell(n)+10$. In (ii)  the corresponding estimate
for $\La_{k_1}\sK^{n,\nu}_{j,y}
[V_{k_2}a]$ is entirely analogous and will be omitted.

We use the Fourier inversion formula for $a$ and for the convolution kernel of $\La_{k_1}$,  write \begin{multline*}
\La_{k_1}\sK^{n,\nu}_{j,y}[\La_{k_2}a](x)
=\frac{1}{(2\pi)^{2d}}\int\vth_n(s) \iint
\beta_{k_1}(\xi) \beta_{k_2} (\eta) \,\widehat a(\eta)\,
\times
\\
 \Big[\int_w e^{i (\inn{x-w}{\xi}+ \inn{sw+(1-s)y}{\eta})} K^{n,\nu}_j(w-y) \, dw\,
\Big]\,d\xi\,
d\eta \, ds\,,
\end{multline*}
and  integrate by parts with respect to $w$ and  $\xi$. The integral can then be rewritten as
\footnote{Thanks to Xudong Lai who  pointed out
an error in the original version of this formula.}
\begin{multline*}
\frac{1}{(2\pi)^{2d}}\int\vth_n(s) \int
 \,\beta_{k_2}(\eta)
\widehat a(\eta)\,\int
\Big[\int e^{i (\inn{x-w}{\xi}+ \inn{sw+(1-s)y}{\eta})} \times \\
\frac{(I-2^{-2k_1}\Delta_\xi)^{N_1} [\beta_{k_1}(\xi)|\xi-s\eta|^{-2N_2}]
(-\Delta_w)^{N_2} K^{n,\nu}_j(w-y)}
{(1+2^{-2k_1}|x-w|^2)^{N_1}} 
 \, dw\,
\Big]\,d\xi\,
d\eta \, ds\,,
\end{multline*}
and we choose $N_1=[d/2]+1$.
Note that
for $s\in supp(\vth_n)$,
 $$|\xi-s\eta|\gc \cC(k_1,k_2,n)
:=\begin{cases} 2^{-k_2-\ell(n)} &\text{  if $k_1> k_2+\ell(n)+10$\,,}
\\2^{-k_1-2} &\text{ if  $k_1<k_2-10$\,.} 
\end{cases}
$$
Now $(2^{-k_1}\partial_\xi)^{N_3} \beta_{k_1} =O(1)$ and thus  one computes 
$$
\big|(I-2^{-2k_1}\Delta_\xi)^{N_1} 
[\beta_{k_1}(\xi)|\xi-s\eta|^{-2N_2}]\big|\lc
[\cC(k_1,k_2,n)]^{-N_2}\,.
$$ Moreover
\begin{align*}\|(-\Delta_w)^{N_2} K^{n,\nu}_j\|_1 &\lc 2^{-2N_2j}
(2^{2N_2 n\gamma}+ 2^{2N_2\ell_\eps(n)}) 2^{-n\gamma(d-1)}
\\
&\lc 2^{-n\gamma(d-1)} 2^{2N_2(n\gamma-j)}\,;
\end{align*}

We 
integrate in $\eta$ and use that the size of the support of $\beta_{k_2}$ is $2^{-k_2 d}$.  Then we integrate in $x,\xi$  and use that
$$\int_{\supp(\beta_{k_1})}\int (1+2^{-2k_1}|x-w|^2)^{-N_1} \, dx\,d\xi =O(1)\,.$$
Using \eqref{FTaest} we then get
\begin{align*}
&\big\|\La_{k_1}\sK^{n,\nu}_{j,y}[\La_{k_2}a]\big \|_1\,\lc_{N_2} \,2^{-k_2 d}
\|\widehat a\|_\infty \|(-\Delta)^{N_2} K^{n,\nu}_j\|_1 [\cC(k_1,k_2,n)]^{-2N_2}
\\&\lc_{N_2}\begin{cases}
2^{d\ell(n)-n\gamma(d-2)} 2^{(2N_2-d) (k_2-j+\ell(n)+n\gamma)} \text{ if $k_1>k_2+\ell(n)
+10$\,,}
\\
2^{-n\gamma(d-2)} 2^{(2N_2-d) (k_1-j+n\gamma)} 2^{(k_1-k_2)d}\,\text{ if $k_1<k_2-10$\,.}
\end{cases}
\end{align*}
If we put $N=2N_2-d$ this gives the asserted  bound for
$\|\La_{k_1}\sK^{n,\nu}_{j,y}[\La_{k_2}a]\big \|_1$. For $k_1<k_2-10$ the corresponding expression with
$\La_{k_2}$ replaced by $V_{k_2}$ is estimated in exactly the same way.
\end{proof}

\begin{lemma} \label{VandLPsamek}
 Let
$k_2-10\le k_1\le k_2+\ell(n)+10$. Then
\begin{multline*}\|\La_{k_1}(I-\fS^{n,\nu})\sK^{n,\nu}_{j,y}[\La_{k_2}a]\|_1 \\
\le C_N
2^{-n\gamma(d-1)} \min\{ 1, n^{2(N+d)/\eps}
2^{(d+3)n\gamma} 2^{(k_1-j+n\gamma)N}\}\,
\end{multline*}
for every  $y\in Q$.
\end{lemma}
\begin{proof}
We may again assume that  \eqref{FTaest} holds.
Define the  convolution operator $S_{n,\nu}$ by
$$\widehat {S^{n,\nu} g}(\eta)= \phi(2^{n\gamma} n^{-2} \inn{\nu}{\tfrac{\eta}{|\eta|}})
\widehat g(\eta)$$ and split $a=  S^{n,\nu} a+(I-S^{n,\nu} )a$.
We shall prove the  following  estimates, 
\begin{multline}
 \label{Snnucontr}\|
\La_{k_1}(I-\fS^{n,\nu})\sK^{n,\nu}_{j,y}[\La_{k_2}S^{n,\nu}a]\|_1\,\\
\le C_N\,
n^{(2\eps^{-1}-4)(N+d)}2^{4n\gamma}2^{(k_1-k_2)d}2^{(k_1-j+n\gamma)N}
\end{multline}
and
\Be \label{IminusSnnucontr}\|
\La_{k_1}(I-\fS^{n,\nu})\sK^{n,\nu}_{j,y}[\La_{k_2}(I-S^{n,\nu})a]\|_1\,
\le C_N\, n^{-5d}2^{4n\gamma}  2^{(k_2-j+n\gamma)N}\, ,
\Ee
which  imply the somewhat weaker estimate asserted in the lemma.

\noi{\it Proof of \eqref{Snnucontr}.} Set
$$b_{k_1,n,\nu}(\xi)=
\beta_{k_1}(\xi)
\big(1- \phi(2^{n\gamma} n^{-5} \inn{\nu}{\tfrac{\xi}{|\xi|}})\big)
$$ and write
\begin{multline*}
(2\pi)^{2d}\La_{k_1}(I-\fS^{n,\nu})\sK^{n,\nu}_{j,y}[\La_{k_2}S^{n,\nu}a](x)\,=\\ \int\vth_n(s) \iint b_{k_1,n,\nu}(\xi)
\beta_{k_2} (\eta) \phi(2^{n\gamma} n^{-2} \inn{\nu}{\tfrac{\eta}{|\eta|}})
\,\widehat a(\eta)
\\ \times
\Big[\int_w e^{i (\inn{x-w}{\xi}+ \inn{sw+(1-s)y}{\eta})} K^{n,\nu}_j(w-y) \, dw\,
\Big]\,d\xi\,
d\eta \, ds\,.
\end{multline*}
If $(\xi,\eta)$ is in the support of the amplitude then for $n>10^{10}$
\begin{align}\notag
&\big|\inn{\xi-s\eta}{\nu}\big| \ge |\xi| \big |\inn{\tfrac{\xi}{|\xi|}}{\nu}\big|
-|\eta| \big |\inn{\tfrac{\eta}{|\eta|}}{\nu}\big|
\\
\notag
&\ge \,|\xi| \big( 2^{-n\gamma-1} n^5 - 2^{|k_1-k_2|+2} 2^{-n\gamma} n^2\big) 
\\
\label{lowerbddenom}
&\,\ge \,|\xi|2^{-n\gamma-1} ( n^5- 8\cdot 2^{\ell(n)+10}n^2)
\ge  \,2^{-k_1-n\gamma } n^5\,.
\end{align}
Now we can integrate by parts as in the proof of Lemma \ref{VandLP},
except we use the directional derivative $\inn {\nu}{\nabla_w}$ instead of
$\Delta_w$.
The above integral is then estimated by
\begin{multline*}
\int \!\!\!\iiint |\beta_{k_2} (\eta)|\,|\widehat a(\eta)|
\big |\phi(2^{n\gamma} n^{-2} \inn{\nu}{\tfrac{\eta}{|\eta|}})\big|
\\ \times
\frac{\big|(I-2^{-2k_1}\Delta_{\xi})^{N_1}
\big[\frac{b_{k_1,n,\nu}(\xi)}{\inn {\xi-s\eta}{\nu}^{N_2}} \big]}
{(1-2^{-2k_1}|x-w|^2)^{N_1}}
|\inn{\nu}{\nabla_w}^{N_2} K^{n,\nu}_j(w-y) |
\, dw\,d\xi\,
d\eta \, ds\,.
\end{multline*}
 Observe that
$$
\big|\partial_\xi^{N_3}b_{k_1,n,\nu}(\xi) \big|
\le C_{N_1} (2^{n\gamma}n^{-5})^{N_3}2^{k_1N_3}
$$
and thus
\Be \label{bk1nnu}
\big|(I-2^{-2k_1}\Delta_\xi)^{N_1}\big[\frac{b_{k_1,n,\nu}(\xi)}
{\inn {\xi-s\eta}{\nu}^{N_2}} \big]
 \big|
\le C_{N_1} (2^{n\gamma}n^{-5})^{2N_1}(2^{-(k_1+n\gamma)}n^{5})^{-N_2}.
\Ee
Moreover,
$$
\big\|\inn{\nu}{\nabla_w}^{N_2} K^{n,\nu}_j\big\|_1\le C_{N_2}
2^{(\ell_\eps(n) -j)N_2}
2^{-n\gamma(d-1)}\,.
$$
We assume $2N_1>d$, integrate in $x$ and $\xi$, and use
\eqref{lowerbddenom}. Then we obtain
\begin{multline*}
 \|
\La_{k_1}(I-\fS^{n,\nu})\sK^{n,\nu}_{j,y}[\La_{k_2}S^{n,\nu}a]\|_1\,\\
\lc_{N_1,N_2}\,
(2^{2n\gamma}n^{-5})^{2N_1} \|\widehat a\|_\infty 2^{-k_2d}
\frac{2^{(\ell_\eps(n)-j)N_2}2^{-n\gamma(d-1)}}{(2^{-k_1-n\gamma}n^5)^{N_2}}.
\end{multline*}
We use  \eqref{FTaest} and that the support
of 
$\eta\mapsto
\beta_{k_2} (\eta)$
has measure $O(2^{-k_2d})$. 
Thus the expression in the
previous display can be crudely estimated by
$$C_{N_1,N_2} n^{(2\eps^{-1}-4)N_2-10N_1}2^{n\gamma(2N_1-d+2)} 2^{(k_1-k_2)d}
2^{(k_1-j+n\gamma)(N_2-d)}
$$
and,  if we chose the integer $N_1\in
\{\frac{d+1}2, \frac{d+2}2\}$
and $N=N_2-d$ we obtain \eqref {Snnucontr}.

\medskip

\noi{\it Proof of \eqref{IminusSnnucontr}.} Set
$$\widetilde b_{k_2,n,\nu}(\eta)=
\beta_{k_2} (\eta) (1-\phi(2^{n\gamma} n^{-2} \inn{\nu}{\tfrac{\eta}{|\eta|}}))
$$
and write
\begin{multline*}
(2\pi)^d\La_{k_1}(I-\fS^{n,\nu})\sK^{n,\nu}_{j,y}[\La_{k_2}(I-S^{n,\nu})a](x)\,\\
\,=\,
\int K^{n,\nu}_j(w-y)
\iint b_{k_1,n,\nu}(\xi) \widetilde b_{k_2,n,\nu}(\eta)
\,\widehat a(\eta)
\\ \times
\Big[\int\vth_n(s)  e^{i (\inn{x-w}{\xi}+ \inn{sw+(1-s)y}{\eta})} \, ds
\Big]\,d\xi\,
d\eta \, dw\, .
\end{multline*}
Now
if $w-y\in \supp (K^{n,\nu}_j)$ then
$\big|\tfrac{w-y}{|w-y|}-\nu\big| \le  2^{-n\gamma}$
and if $\eta \in \supp (\widetilde b_{k_2,n,\nu})$ we get
$$|\inn{w-y}{\eta}|
\ge |w-y|\,\big(\inn{\nu}{\eta}-|\eta|2^{-n\gamma}\big)\ge
|w-y|\,|\eta|2^{-n\gamma} (\tfrac {1}{2}n^2 -1)$$
and hence
\Be\label{scprdctwitheta}
|\inn{w-y}{\eta}|
\ge 2^{j-k_2-n\gamma-4} n^2.
\Ee
Integration by parts with respect to $s$ yields
\begin{multline*}
(2\pi)^d\La_{k_1}(I-\fS^{n,\nu})\sK^{n,\nu}_{j,y}[\La_{k_2}(I-S^{n,\nu})a](x)
\,=\, \\
\int K^{n,\nu}_j(w-y)
\iint \widehat a(\eta)\widetilde b_{k_2,n,\nu}(\eta)
\frac{(I-2^{-2k_1}\Delta_\xi)^{N_1}  b_{k_1,n,\nu}(\xi) }{
(1+2^{-2k_1}|x-w|^2)^{N_1}} e^{i (\inn{x-w}{\xi}+ \inn{y}{\eta})}
\\ \times
\Big[\int\vth_n^{(N_3)}(s) \frac{i^{N_3} e^{is \inn{w-y}{\eta}}}
{\inn{w-y}{\eta}^{N_3}} \, ds
\Big]\,d\xi\,
d\eta \, dw\, .
\end{multline*}
We apply this with  $N_1>d/2$ and, using  \eqref{FTaest},
\eqref{bk1nnu}, and 
\eqref{scprdctwitheta},  obtain 
\begin{align*}
&\|\La_{k_1}(I-\fS^{n,\nu})\sK^{n,\nu}_{j,y}\La_{k_2}(I-S^{n,\nu})a\|_1
\\
&\lc_{N_1, N_3}(2^{n\gamma}n^{-5})^{2N_1}\|K_j^{n,\nu}\|_1
\frac{\|\vth_n^{(N_3)} \|_1}{(2^{j-k_2-n\gamma-4} n^2)^{N_3}}
2^{-k_2d}\|\widehat a\|_\infty
\\
&\lc_{N_1,N_3} n^{-2-10N_1} 2^{n\gamma(2N_1-d+2)} 2^{(k_2-j+n\gamma)(N_3-d)}\,.
\end{align*}
Inequality \eqref{IminusSnnucontr} follows if we choose $N=N_3-d$ and  $N_1\in
\{\frac{d+1}2, \frac{d+2}2\}$.
\end{proof}

\begin{proof} [Proof of Proposition \ref{L1prop}, conclusion]  Let,
for fixed $n,\nu,j$ and for a fixed cube $Q\in\fQ_\la$ with $L(Q)=j-n$,
$$I_{k_1}=\widetilde \La_{k_1}(I-P_{j-n+\ell(n)}) \La_{k_1}(I-\fS_{n,\nu})
\Big[
\int b_Q(y) \sK_{j,y}^{n,\nu}[V_{j-n+\ell(n)}a] dy\Big]\,,
$$ and
$$
II_{k_1,k_2}= \widetilde \La_{k_1}(I-P_{j-n+\ell(n)})
\La_{k_1}(I-\fS_{n,\nu})
\Big[
\int b_Q(y) \sK_{j,y}^{n,\nu}[\La_{k_2}a](x) dy\Big]\,.
$$
By \eqref{split}, \eqref{asplit} it is enough to show that
\Be \label{ksums}\sum_{k_1} \|I_{k_1}\|_1 + \sum_{k_1}\sum_{k_2< j-n+\ell(n)} \|II_{k_1,k_2}\|_1 \lc n^{-2}2^{-\gamma n(d-1)}\,\|b_Q\|_1\,.
\Ee
We have
\Be\label{sectorialL1est}
\|\La_{k_1}(I-\fS_{n,\nu})\|_{L^1\to L^1} \le C
\Ee uniformly in $n,\nu, k_1$, and using the support and cancellation properties of the kernel of $I-P_{j-n+\ell(n)}$ we also have
\Be \label{cancleft} \|\widetilde \La_{k_1}(I-P_{j-n+\ell(n)})\|_{L^1\to L^1}
\lc \min \{1, 2^{j-n+\ell(n)-k_1}\}\,.
\Ee
Lemma \ref{firstcancel} together with \eqref{cancleft}, \eqref{sectorialL1est}
immediately gives 
\Be \label{cancellationsum}\sum_{k_1\ge j-n+\ell(n)-10}
\|I_{k_1}\|_1 \lc n^{-2}2^{-\gamma n(d-1)}\,\|b_Q\|_1.\Ee

It remains to  verify that the other terms satisfy better bounds, namely
\begin{multline} \label{ksumsbetter}\sum_{k_1< j-n+\ell(n)-10} \|I_{k_1}\|_1 + \sum_{k_1}\sum_{k_2< j-n+\ell(n)} \|II_{k_1,k_2}\|_1 \\
\lc C_N n^{A_1N}2^{A_2n} 2^{n(\gamma-1)N}\,\|b_Q\|_1
\end{multline}
for all $N$, and suitable $A_1\le 10d/\eps$, $A_2\le 10$. Choose $N=100d$.
Taking into account that $\gamma\le 9/10$ one may check  that  the bound in \eqref{ksumsbetter} is $\lc n^{-2} 2^{-n\gamma(d-1)}\|b_Q\|_1$ for all $n$ 
with $n^{-1}\log n\le 
10^{-4}\eps/d$, which is satisfied for $n>n(\eps)$.

For the terms involving $I_{k_1}$,  with $k_1\ge j-n+\ell(n)+10$ we get by the
second estimate in part (ii) of Lemma \ref{VandLP}, with $k_2=j-n+\ell(n)$,
\begin{align*}&\sum_{k_1< j-n+\ell(n)-10}
\|I_{k_1}\|_1 \\&\lc_N 2^{-n\gamma (d-2)}
\sum_{k_1< j-n+\ell(n)-10}
2^{(k_1-j+n-\ell(n))d} 2^{(k_1-j+n\gamma)N}\,\|b_Q\|_1
\\&\lc_N 2^{-n\gamma (d-2)} (2^{n(\gamma-1)}n^2)^N
\,\|b_Q\|_1.
\end{align*}
Next consider  $\sum_{k_1,k_2}\|II_{k_1,k_2}\|_1$ where the $k_2$-summation is extended
over $k_2<j-n+\ell(n)$.
For $k_1\ge j-n+\ell-10$
we can sum a geometric series in $k_1$, with a uniform bound, due to \eqref{cancleft}.
By Lemma \ref{VandLP}, part (i)
\begin{align*}
&\sum_{(k_1,k_2): \substack{
k_1\ge j-n+\ell(n)-10\\k_2<\min \{k_1-\ell(n)-10, j-n+\ell(n)\}}}
\|II_{k_1,k_2}\|_1
\\
&\qquad\lc
2^{-n\gamma(d-2)} n^{2d+2N} \sum_{k_2<j-n+\ell(n)}
2^{(k_2-j+n\gamma)N} \|b_Q\|_1
\\&\qquad\lc
2^{-n\gamma(d-2)} n^{2d+4N} 2^{n(\gamma-1)N}\|b_Q\|_1\,,
\end{align*}
and
by  Lemma \ref{VandLPsamek}
\begin{align*}
&\sum_{
\substack{k_1\ge j-n+\ell(n)-10\\ k_1-\ell(n)-10\le k_2\le k_1+10\\ k_2<j-n+\ell(n)}}
\|II_{k_1,k_2}\|_1
\\
&\qquad\lc \|b_Q\|_1\ell(n) n^{2(N+d)/\eps} 2^{4n\gamma}
\sum_{k_1\le j-n+2\ell(n)+10} 2^{(k_1-j+n\gamma)N}
 \\&\qquad\lc \|b_Q\|_1
\log (n) n^{2(N+d)(\eps^{-1}+2)} 2^{n(\gamma-1)N}.
\end{align*}
The case $k_2>k_1+10$ does not occur when $k_1\ge j-n+\ell(n)-10$
because of the restriction
$k_2<j-n+\ell(n)$.  Thus in all   cases of  \eqref{ksumsbetter} which involve 
the restriction $k_1\ge j-n+\ell(n)-10$ we obtain the required estimate.

Now sum the terms $\|II_{k_1,k_2}\|_1$ with $k_1< j-n+\ell-10$.
By Lemma \ref{VandLP}, part (i)
\begin{align*}
&\sum_{(k_1,k_2):\substack{k_1< j-n+\ell(n)-10\\k_2<k_1-\ell(n)-10}}
\|II_{k_1,k_2}\|_1
\\
&\lc n^{2d+2N} 2^{-n\gamma(d-2)}
\sum_{(k_1,k_2):\substack{k_1< j-n+\ell(n)-10\\k_2<k_1-\ell(n)-10}}
2^{(k_2-j+n\gamma)N} \|b_Q\|_1\,,
\\
&\lc n^{2d+2N} 2^{-n\gamma(d-2)} 2^{n(\gamma-1)N} \|b_Q\|_1\,,
\end{align*}
by Lemma \ref{VandLP}, part (ii)
\begin{align*}
&\sum_{(k_1,k_2):\substack{k_1< j-n+\ell(n)-10\\k_1+10<k_2<j-n+\ell(n)-10}}
\|II_{k_1,k_2}\|_1
\\
&\lc 2^{-n\gamma(d-2)} \sum_{k_1< j-n+\ell(n)-10}   2^{(k_1-j+n\gamma)N}
\sum_{k_2>k_1+10}2^{(k_1-k_2)d}\|b_Q\|_1\,,
\\& \lc n^{2N} 2^{-n\gamma(d-2)} 2^{n(\gamma-1)N} \|b_Q\|_1\,,
 \end{align*}
and finally,
by  Lemma \ref{VandLPsamek},
\begin{align*}
&\sum_{(k_1,k_2):\substack {k_1< j-n+\ell(n)-10\\k_1-\ell(n)-10\le k_2\le k_1+10}}
\|II_{k_1,k_2}\|_1
\\
&\lc \log(n) n^{2(N+d)/\eps} 2^{4n\gamma} 
\sum_{k_1\le j-n+\ell(n)}
2^{(k_1-j+n\gamma )N} \|b_Q\|_1\,
\\
&\lc n^{2(N+d)(\eps^{-1}+1)}  2^{4n\gamma} 
 2^{n(\gamma-1)N}  \|b_Q\|_1\,.
\end{align*}
This finishes the proof of \eqref{ksumsbetter}.
\end{proof}

\section{Proof of Proposition \ref{L2prop}}
\label{proofofL2prop}
We use a slightly modified version of an 
 argument in \cite{see}. The main observation is that, for fixed $n> 0$,
we have
\Be\label{overlap}\sup_{\xi\neq 0}\sum_{\nu\in \Theta_n}
|\phi(2^{n\gamma}n^{-5}\inn{\nu}{\tfrac{\xi}{|\xi|}})| \lc
2^{n\gamma (d-2)}n^5
.\Ee
To see this it suffices, by  homogeneity,
 to  take the supremum over all $\xi\in S^{d-1}$.
Now if $|\xi|=1$ and $\phi(2^{n\gamma}n^{-5}\inn{\theta}{\xi})\neq 0$
then  the distance of $\nu$ to the hyperplane
$\xi$ is at most $C n^5 2^{-n\gamma}$
and since the vectors in $\Theta_n$  are $c2^{-n \gamma}$-separated
there are $O(2^{n\gamma (d-2)}n^5)$ such vectors, hence \eqref{overlap} holds.

From  \eqref{overlap} it follows that 
\begin{multline*}\Big\|\sum_{\nu\in \Theta_n} \fS_{n,\nu}
\sum_j(I-P_{j-n+\ell(n)} ) T_j^{n,\nu} B_{j-n}
\Big\|_2^2 \\ \lc  2^{n\gamma (d-2)}n^5
\sum_{\nu\in \Theta_n}\big\|\sum_j(I-P_{j-n+\ell(n)} )
T_j^{n,\nu} B_{j-n}\big\|_2^2
\end{multline*}
and since $\#\Theta_n \lc 2^{n\gamma(d-1)}$ the asserted inequality is a
consequence of
\Be\label{fixednubound}
\Big\|\sum_j(I-P_{j-n+\ell(n)} ) T_j^{n,\nu}
 B_{j-n}\Big\|_2^2 \lc
2^{-2n\gamma(d-1)} \la \|f\|_1
\Ee
for each $\nu\in \Theta_n$.

For the proof of \eqref{fixednubound} the cancellation of $B_{j-n}$ plays no 
role. Let
$$H^{n,\nu}_j(x)= 2^{-jd}  \chi\ci{\tau^{n,\nu}_j}(x).
$$ where $$\tau^{n,\nu}_j= \{x: |\inn{x}{\nu}|\le 2^{j+2}, \,
|x-\inn{x}{\nu}|\le  2^{j+2-\gamma n}\}.$$
Then from \eqref{czsize} we get
$$\big|(I-P_{j-n+\ell(n)} ) T_j^{n,\nu}B_{j-n}(x)| \lc H^{n,\nu}_j*|B_{j-n}|(x).$$
Therefore
\begin{align*}
&\Big\|\sum_j(I-P_{j-n+\ell(n)} ) T_j^{n,\nu}
 B_{j-n}\Big\|_2^2 \\&\lc 2
\sum_j
\int |B_{j-n}(x)|\, \sum_{i\le j} H^{n,\nu}_j*H^{n,\nu}_i
* |B_{i-n}(x)|  \,dx\,.
\end{align*}
Observe that $\|H^{n,\nu}_i\|_1\lc
2^{-id} \meas( \tau^{n,\nu}_i) \lc 2^{-n\gamma(d-1)}$
and thus
$$H^{n,\nu}_j*H^{n,\nu}_i(x) \lc
2^{-n\gamma(d-1)}\,  2^{-jd} 
\chi_{\widetilde\tau^{n,\nu}_j}(x)$$
where $\widetilde \tau^{n,\nu}_j$ is the double
of $\tau^{n,\nu}_j$.
Hence, for each $x\in\bbR^d$, $j\in \bbZ$,
\begin{align*}
\sum_{i\le j} &H^{n,\nu}_j*H^{n,\nu}_i
* |B_{i-n}|(x)
\\
&\lc  2^{-n\gamma (d-1)}  2^{-jd} \sum_{i\le j}\int_{x+\widetilde \tau^{n,\nu}_j}
|B_{i-n}(y)|\,dy
\\
&\lc 2^{-n\gamma(d-1)}  2^{-jd} 
\sum_{i\le j}\sum_{\substack{Q\in \fQ_\la:\\L(Q)=i-n\\Q\cap(x+\widetilde \tau^{n,\nu}_j)\neq\emptyset}}
\int|b_Q(x)|\,dx
\\
&\lc\, 2^{-n\gamma(d-1)} 2^{-jd} 
\la \,\meas(\widetilde \tau^{n,\nu}_j)\,\lc
 2^{-2n\gamma(d-1)} \la\,;
\end{align*}
here we have used $\|b_Q\|_1\lc\la |Q|$, and the disjointness of the interiors of the  cubes $Q$
in $\fQ_\la$.
Thus we get the estimate
$$\Big\|\sum_j(I-P_{j-n+\ell(n)} ) T_j^{n,\nu}
 B_{j-n}\Big\|_2^2 \lc  2^{-2n\gamma(d-1)} \la\,\sum_j\|B_{j-n}\|_1
$$
which yields \eqref{fixednubound}.
\qed

\section{Open problems}
\label{openproblems}

\subsection{{\it Principal value integrals}} \label{pvcommopen}
Let $$\cT_r f(x)\,=\,
 \int_{|x-y|> r}  K(x-y) \int_0^1 a(sx+(1-s) y) \,ds\,
  f(y)\, dy\,.
$$
Our proof shows that the operators $\cT_r$ are of weak type $(1,1)$,
with  uniform bounds; moreover, for $f\in L^1$,
$\cT_r f$ converges in measure to $Tf$ where $T$ is weak type $(1,1)$.
However it is currently open whether the principal value
$\lim_{r\to 0} \cT_r f(x)$ exists for almost every  $x\in \bbR^d$.
By Stein's theorem \cite{stein} this is equivalent to
the open  question whether the maximal singular integral
$\sup_{r>0} |\cT_r f|$
defines an operator of weak type $(1,1)$.

\subsection{{\it Principal value integrals for rough singular convolution operators}}
The question analogous to \ref{pvcommopen} 
 is open  for classical singular integral
operators with rough convolution  kernel $\Omega(y/|y|)|y|^{-d}$ where $\Omega\in L\log L(S^{d-1})$, $d\ge 2$ and
$\int_{S^{d-1}}\Omega(\theta) d\sigma=0$. These operators are known to be of weak type $(1,1)$, \cite{see}, but the a.e. existence of the principal value integrals is open even for $\Omega\in L^\infty(S^{d-1})$.

\subsection{{\it Christ-Journ\'e operators}}
Let $F\in C^\infty(\bbR)$, let $K$ be a
Calder\'on-Zygmund convolution kernel, and let  $a\in L^\infty(\bbR^d)$.
Christ and Journ\'e \cite{ch-j} showed that
the operator  defined for $f\in C^\infty_0(\bbR^d)$  by
$$\cT f(x) = p.v.  \int
F\big(\int_0^1 a(sx+(1-s) y) dt\big)\,K(x-y)
  f(y) dy\,
$$
extends to a bounded operator on $L^p(\bbR^d)$, $1<p<\infty$.
It would be interesting to get the weak type $(1,1)$ inequality
for nonlinear $F$, in dimension  $d\ge 2$.

\end{document}